\documentclass[smallextended]{svjour3}

\smartqed  

\RequirePackage{fix-cm}
\usepackage{graphicx}
\usepackage{psfrag}
\usepackage{subfigure}
\usepackage{url}
\usepackage{color}
\usepackage{cite}
\usepackage{epsfig}
\usepackage{multirow}
\usepackage{amsmath,amssymb}
\usepackage{enumitem}
\usepackage{comment}

\newcommand{\rset}{\mathbb{R}}

\newtheorem{assumption}[theorem]{Assumption}

\usecounter{algorithmenumi}

\providecommand{\norm}[1]{\lVert#1\rVert}

\begin{document}

\title{ Higher-order tensor methods for  minimizing difference of convex functions }
\titlerunning{Higher-order tensor methods for DC minimization}

\author{Ion Necoara } 

\institute{     
           Ion Necoara \at
              Automatic Control and Systems Engineering Department,  National University of Science and Technology Politehnica Bucharest,   Spl. Independentei 313, 060042 Bucharest, Romania and Gheorghe Mihoc-Caius Iacob Institute of Mathematical Statistics and Applied Mathematics of the  Romanian Academy, 050711 Bucharest, Romania. \email{ion.necoara@upb.ro.}
}
\date{Received: \today / Accepted:}

\maketitle


\begin{abstract}
Higher-order tensor methods  were recently proposed for minimizing smooth convex and nonconvex functions. Higher-order algorithms accelerate the convergence
of the classical first-order methods  thanks to the higher-order derivatives used in the updates. The purpose of this paper is twofold. Firstly, to show that the higher-order algorithmic framework  can be generalized and successfully applied to  (nonsmooth) difference of convex functions, namely, those that can be expressed as the difference of two smooth convex functions and a possibly nonsmooth convex one. We also provide  examples when the subproblem can be solved efficiently, even globally.  Secondly, to derive a complete convergence analysis for our higher-order difference of convex functions  (HO-DC) algorithm. In particular, we prove that any limit point of the HO-DC iterative
sequence is a critical point of the problem under consideration, the corresponding objective value is monotonically decreasing and the minimum value of the  norms of its subgradients converges globally to zero at a sublinear rate. The  sublinear or linear convergence rates of the iterations are obtained under the Kurdyka-Lojasiewicz property. 

\end{abstract}


\section{Introduction}
In this paper we consider minimizing the difference of convex (DC) functions:
\begin{equation}
\label{eq:problem}
\min_{x \in \rset^n}   F(x):= f(x) + \psi(x) - g(x), 
\end{equation}
where  $\psi$ is proper lower semicontinuous  (possibly nondifferentiable) convex function   on the closed convex domain $\text{dom} \, \psi$, while the functions $f$ and $g$ are convex and $p$ and $q$ times continuously differentiable on  $\text{dom} \, \psi$, respectively, with $p,q$ positive intergers. Obviously, $\text{dom} \, F = \text{dom} \, \psi$.  The class of DC functions is very broad, and it includes many important classes of nonconvex functions, such as twice continuously differentiable functions on compact convex sets and multivariate polynomial functions \cite{AhmHal:18,ArtVuo:20}.  For optimization problem \eqref{eq:problem} the first-order necessary optimality conditions at the point of (local) minimum $x^* \in \text{dom} \, \psi$ can be written as follows \cite{Mor:06}:
\begin{align}
\label{eq:opt1}
0 \in  \nabla f(x^*)  - \nabla g(x^*) + \partial \psi (x^*).
\end{align}

\noindent Several algorithms with convergence guarantees have been developed for solving the problem \eqref{eq:problem}. The most well-known method is the difference of convex functions algorithm (DCA), which, in the simplified form, it linearly approximates the concave part of the objective function, $g$, in  \eqref{eq:problem} at the current point and then minimises the resulting convex approximation to the DC function to find the next iteration, without recourse to a line search \cite{LePha:18,LeHuy:18,TaoLe:97}.  Several  algorithms have been also proposed to accelerate the convergence of DCA. For example,  \cite{ArtFle:18,ArtVuo:20}  propose  an algorithm based on a combination of DCA descent direction with a line search step and convergence  is proved under the  Kurdyka-Lojasiewicz property of the objective function. Another variant is the proximal DCA, which adds a quadratic proximal term to the  objective of the convex optimization subproblem  \cite{AnNam:17,BanBot:18,MouMai:06}.  Note that all these  methods are \textit{first-order algorithms}, and despite their empirical success  to solve difficult optimization problems, their convergence speed is known to be slow.

\medskip 

\noindent A natural way to ensure faster convergence rates is to increase the power of the oracle, i.e., to use higher-order information (derivatives) to build a higher-order (Taylor) model. For example, \cite{NesPol:06} derives the first global convergence rate of cubic regularization of Newton method for unconstrained smooth minimization problems with the hessian Lipschitz continuous (i.e., using second-order oracle).  Higher-order methods have become recently popular due to their performance in dealing with ill conditioning and having fast rates of convergence. However, the main obstacle in the implementation of these (higher-order) methods lies in the complexity of the corresponding model approximation formed by a high-order multidimensional polynomial, which may be difficult to  handle and minimize (see for example \cite{BiGaMa:17,cartis2017}). Nevertheless,  for convex  smooth functions  \cite{Nes:20} proved that a regularized Taylor approximation is also convex provided that the regularization parameter is sufficiently large. This observation opens the door for using higher-order Taylor approximations to different structured problems (see, for example, \cite{DoiNes:21,GraNes:20,NecDan:20,NabNec:22}). \textit{However, to the best of our  knowledge, there are not yet methods for minimizing  the difference of convex functions of the form \eqref{eq:problem} using higher-order information with complexity guarantees. This paper is the first to develop a higher-order method for solving  DC problems.} At each iteration our method approximates the two smooth parts of the objective function with higher-order Taylor approximations and add proper regularization terms, leading to a higher-order DC algorithm. We also present convergence guarantees for our new algorithm.

\medskip 

\noindent \textit{Contributions.}  Our main contributions are as follows:

(i) We develop a new higher-order tensor  method for solving difference of convex functions as given in problem \eqref{eq:problem}, called \textit{HO-DC}. Our algorithmic framework is flexible in the sense that we can approximate both terms in the objective function  with higher-order Taylor approximations of different degrees  (i.e., we can approximate function $f$ with a Taylor approximation of degree $p$ and  $g$ with a Taylor  of degree $q$). An adaptive variant is also presented.

(ii) We derive a complete convergence analysis  for our  HO-DC algorithm. More precisely, for the general problem, we show that any limit point of the HO-DC iterative sequence is a critical point, the corresponding objective value is monotonically decreasing and the minimum value of the  norms of its subgradients converges globally to zero at a rate of order $\mathcal{O}(k^{-\frac{2\min(p,q)}{p+q+2}})$, where $k$ is the iteration counter and $p$ and $q$ are the degrees of the Taylor approximations for objectives $f$ and $g$ , respectively. When the objective function satisfies the Kurdyka-Lojasiewicz property (e.g., the objective function is semi-algebraic), we prove that the whole sequence generated by HO-DC algorithm converges and  derive  (linear) sublinear convergence rates in the iterates (depending on the parameter of the KL property). 

(iii) The subproblem we need to solve at each iteration of HO-DC  is usually nonconvex and it can have local minima. However, we show  for $p, q \in\{1,2\}$ that our approach is implementable, since this subproblem is equivalent to minimizing an explicitly written one-dimensional convex function over a convex set that can be solve using efficient convex optimization tools. We believe that this is an important step towards practical implementation of higher-order (such as cubic regularized Newton type) methods in DC programming. 

\medskip 

\noindent  Besides providing a unifying  \textit{global} convergence   analysis of  higher-order  methods for DC problems, in special cases, where complexity bounds are known for some particular   algorithms, our convergence results recover the  existing bounds. For example, for $p=q=1$, we recover (see Theorem \ref{th:4}) and even extend (see Theorem \ref{th:2}) the convergence results obtained in \cite{ArtFle:18,ArtVuo:20,LePha:18,LeHuy:18,AnNam:17,BanBot:18} for the (proximal) DC algorithms.

\section{Notations and preliminaries}

\noindent In what follows, $ \mathbb{R}^n $ denotes the finite-dimensional Euclidean space endowed with the standard inner product $\langle s,x\rangle = s^{T}x$ and the corresponding norm $\|x\| =(x^Tx)^{1/2}$ for any $s,x\in  \mathbb{R}^n$.   For a twice differentiable function $\phi$ on a convex and open domain $\text{dom}\;\phi \subseteq \mathbb{R}^n$, we denote by $\nabla \phi(x)$ and $\nabla^{2} \phi(x)$ its gradient and hessian evaluated at  $x\in \text{dom}\;\phi$, respectively.  Throughout the paper, we consider $p$ a positive integer.  In what follows, we often work with directional derivatives of function $\phi$ at $x$ along directions $h_{i}\in \mathbb{R}^n$ of order $p$, $D^{p} \phi (x)[h_{1},\cdots,h_{p}]$, with $i=1:p$. For example, the directional derivative of order $1$ of the function $\phi$ is defined in the usual way: $D \phi (x)[h] = \lim_{\alpha \to 0} (\phi(x + \alpha h) -\phi(x))/\alpha$. If all the directions $h_{1},\cdots,h_{p}$ are the same, we use the notation
$ D^{p} \phi(x)[h]$, for $h\in\mathbb{R}^n.$
Note that if $\phi$ is $p$ times differentiable, then $D^{p} \phi (x)$ is a symmetric $p$-linear form and its norm is defined as \cite{Nes:20}:
\begin{align*}
\norm{D^{p} \phi (x)}= \max\limits_{h\in\mathbb{R}^n} \left\lbrace D^{p} \phi (x)[h]^{p} :\norm{h}\leq 1 \right\rbrace.  
\end{align*}

\noindent Further, the Taylor approximation of order $p$ of $\phi$ at $x\in \text{dom}\;\phi$ is denoted:
\begin{align*}
    T_p^{\phi}(y; x)= \phi(x) + \sum_{i=1}^{p} \frac{1}{i !} D^{i} \phi(x)[y-x]^{i} \quad \forall y \in \mathbb{R}^n.
\end{align*}

\noindent Let $\phi: \mathbb{R}^n\mapsto\bar{\mathbb{R}}$ be a $p$ differentiable function on $\text{dom}\;\phi$. Then, the $p$ derivative is Lipschitz continuous if there exist a constant $L_p^{\phi} > 0$ such that:
   	\begin{equation} \label{eq:001}
    	\| D^p \phi(x) - D^p \phi(y) \| \leq L_p^{\phi} \| x - y \| \quad  \forall x,y \in \text{dom}\;\phi.
    	\end{equation}  

\noindent Let us give now a nontrivial   example of a function having the $p$  derivative Lipschitz continuous. 
\begin{example}
	\label{expl:2}
	For given $a_{i} \in \mathbb{R}^n, 1 \leq i \leq m,$ consider the log-sum-exp function:
	\[  \phi(x)=\log \left(\sum_{i=1}^{m} e^{\left\langle a_{i}, x\right\rangle}\right), \quad x \in \mathbb{R}^n.  \]
	Then, the Lipschitz continuous condition	\eqref{eq:001} holds  for $p \in \{1, 2, 3\}$, see  \cite{NecDan:20}. Note that for $m=2$ and  $a_1 =0$, we recover the  expression of logistic regression function, which is a loss function widely used in machine learning. \qed 
\end{example}
     
\noindent It is well known that if \eqref{eq:001} holds,  then the residual between the function  and its Taylor approximation can be bounded \cite{Nes:20}:
    \begin{equation}\label{eq:TayAppBound}
    \vert \phi(y) - T_p^{\phi}(y;x) \vert \leq  \frac{L_p^{\phi}}{(p+1)!} \norm{y-x}^{p+1} \quad  \forall x,y\in \text{dom}\;\phi.
    \end{equation}

\noindent  If $p \geq 2$, we also have the following inequalities valid for all $ x,y\in \text{dom}\;\phi$:
    \begin{align} \label{eq:TayAppG1}
    &\| \nabla \phi(y) - \nabla T_p^{\phi}(y;x) \| \leq \frac{L_p^{\phi}}{p!} \|y - x \|^p, \\
    \label{eq:TayAppG2}
    &\|\nabla^2 \phi(y) - \nabla^2 T_p^{\phi}(y;x) \| \leq \frac{L_p^{\phi}}{(p-1)!} \|y - x\|^{p-1},
    \end{align}

\noindent where the norm defined in \eqref{eq:TayAppG2} corresponds to the spectral norm of a symmetric matrix. Next, we  provide the definition of subdifferential of a function \cite{Mor:06}.

\begin{definition}
 Let $\phi: \mathbb{R}^n \to \bar{\mathbb{R}}$ be a proper lower semicontinuous function. For a given $x \in \text{dom} \; \phi$, the Fr$\acute{e}$chet subdifferential of $\phi$ at $x$, written $\widehat{\partial}\phi(x)$, is the set of all vectors $\phi^{x}\in\mathbb{R}^n$ satisfying:
\begin{equation*}
\lim_{x\neq y}\inf\limits_{y\to x}\frac{\phi(y) - \phi(x) - \langle \phi^{x}, y - x\rangle}{\norm{x-y}}\geq 0.
\end{equation*}
When $x \notin\text{dom} \; \phi$, we set $\widehat{\partial} \phi(x) = \emptyset$. The limiting-subdifferential, or simply the subdifferential of $\phi$ at $x\in \text{dom} \, \phi$, written $\partial \phi(x)$, is defined as \cite{Mor:06}:
\begin{align*}
\partial \phi(x) \!=\! \left\{ \phi^{x}\in \mathbb{R}^n\!\!: \exists x^{k}\to\! x, \phi(x^{k})\to\! \phi(x) \; \text{and} \; \exists \phi^{x}_{k}\in\widehat{\partial} \phi(x^{k}) \; \text{s.t.} \;  \phi^{x}_{k} \to \phi^{x}\right\}. 
\end{align*} 
\end{definition}

\noindent If the function $\phi$ is proper lower semicontinous and convex, then  $\partial \phi(x) = \widehat{\partial}\phi(x)$.    Denote $S_\phi(x) := \text{dist}(0,\partial \phi(x))$. A vector $x^*$ is called a stationary point of the function $\phi$ if $0 \in \partial \phi(x^*)$. It is known that any (local) minima of a function $\phi$ is a stationary point \cite{Mor:06}. The function $\phi$ is said to be coercive if
$\phi(x) \to \infty$ whenever $x \to \infty$.  Next, we recall the definition of a function satisfying the \textit{Kurdyka-Lojasiewicz} (KL) property (see \cite{AtBol:09,AtBol:13,BolSab:14} for more details). 
\begin{definition}
\label{def:kl}
\noindent A proper lower semicontinuous  function $\phi: \mathbb{R}^n \to \bar{\mathbb{R}}$ satisfies the  \textit{Kurdyka-Lojasiewicz (KL)} property on the compact connected set $\Omega \subseteq \text{dom} \; \phi$ on which $\phi$ takes a constant value $\phi_*$ if there exist $\delta, \epsilon >0$ s.t. one has:
\begin{equation}\label{eq:kl_0}
\kappa' (\phi(x) - \phi_*) \cdot  S_{\phi}(x)  \geq 1  \quad   \forall x\!:  \emph{\emph{dist}}(x, \Omega) \leq \delta, \;  \phi_* < \phi(x) < \phi_* + \epsilon,  
\end{equation}
where $\kappa: [0,\epsilon] \mapsto \mathbb{R}_+$ is  concave differentiable function with  $\kappa(0) = 0$ and $\kappa'>0$. 
\end{definition} 

\noindent If $\phi$ is semi-algebraic,  there  exist $r >1$ and $\sigma_r>0$ such that  $ \kappa$ in Definition \ref{def:kl} is of the form  $\kappa (t) = \sigma_r^{\frac{1}{r}} \frac{r}{r-1} t^{\frac{r-1}{r}}$  \cite{AtBol:09,AtBol:13,BolSab:14}. In this case  the KL property establishes the following local geometry of $f$ around a compact set~$\Omega$:
\begin{equation}\label{eq:kl}
\phi(x) - \phi_*  \leq \sigma_r  S_{\phi}(x)^r \quad   \forall x\!: \;  \text{dist}(x, \Omega) \leq \delta, \; \phi_* < \phi(x) < \phi_* + \epsilon.  
\end{equation}

\noindent  Note that the relevant aspect of the KL property is when $\Omega$ is a subset of stationary points for $f$, i.e.  $\Omega \subseteq \{x: 0 \in \partial \phi(x) \}$, since it is easy to establish the KL property when $\Omega$ is not related to stationary points. The KL property holds for a large class of functions including semi-algebraic functions (e.g., real polynomial functions), vector or matrix (semi)norms (e.g., $\|\cdot\|_p$ with $p \geq 0$ rational number), logarithm functions,  exponential functions and  uniformly convex functions,  see \cite{AtBol:09,AtBol:13,BolSab:14} for a comprehensive list.


\section{Higher-order DC algorithm}
In this section, we present a new  higher-order algorithm for solving  the DC problem \eqref{eq:problem}. We consider the following assumptions for problem \eqref{eq:problem}:
\begin{assumption}
\label{as:1}
The following statements hold:
\vspace{-0.25cm}
\begin{enumerate}
    \item The convex function  $f$ is $p$ times differentiable function  with the $p$th derivative  Lipschitz continuous with constant $L_p^f$ on the closed convex set $\text{dom}\psi$.
    \item The convex function  $g$ is $q$ times differentiable function  with the $q$th derivative  Lipschitz continuous with constant $L_q^g$ on the closed convex set $\text{dom}\psi$.
    \item The proper lower semicontinuous convex function $\psi$ is simple. 
    \item Problem \eqref{eq:problem} has a solution and hence $\inf_{x \in \text{dom}\psi}F(x) \geq F^* > - \infty$.
\end{enumerate}
\end{assumption}

\noindent From Assumption \ref{as:1} and the inequality \eqref{eq:TayAppBound}, we get: 
\begin{align}
\label{eq:tpf}
&    \left |f(y) - T_p^{f}(y;x) \right| \leq \frac{L_p^f}{(p+1)!}\|y - x\|^{p+1}\;\;\;  \forall x, y \in\mathbb{R}^n \\
&  \left |g(y) - T_q^{g}(y;x) \right| \leq \frac{L_q^g}{(q+1)!}\|y - x\|^{q+1}\;\;\;  \forall x, y \in\mathbb{R}^n 
\label{eq:tqg}
\end{align}
Based on these  bounds one can  consider the following upper approximation:
\begin{align}\label{sub:1}
& m_{p}^q(y;x) \nonumber \\
& := T_p^f(y;x) + \frac{M_p}{(p+1)!}\|y - x\|^{p+1} + \psi(y) -  T_q^g(y;x) + \frac{M_q}{(q+1)!}\|y - x\|^{q+1} \nonumber \\
& \geq  F(y) \quad \forall y  \in \text{dom}F, \; M_p > L_p^f, \; M_q > L_q^g. 
\end{align}

\noindent Let us analyze in more depth the approximation model $m_{p}^q(\cdot;x)$ that needs to be minimized at each  current point $x$. First, according to Assumption \ref{as:1}.3, $\psi$ is a \textit{simple} function, that is, the presence of  this function in the approximation model $m_{p}^q(\cdot;x)$ does not add computational difficulties in minimizing it. Second,  note that if we denote with $y^*(x)$ any stationary point for the approximation model $m_{p}^q(y;x)$, i.e., $0 \in \partial m_{p}^q(y^*(x);x)$, then $y^*(x) = x$ if and only if $x$ is a stationary point for the original objective function $F$, i.e., $0 \in \partial F(x)$, respectively. 

\medskip

\noindent Now we are ready to present our new higher-order algorithm for DC programming, \textit{called HO-DC}:
\begin{center}
\noindent\fbox{
\parbox{11.5cm}{
\textbf{ Algorithm HO-DC}
\begin{enumerate}
\item[] Choose $x_0  \in \text{dom}\psi$, $M_p > L_p^f$ and $M_q > L_q^g$.
\item[] For $k\geq 0$ do: 
\item[] Compute $x_{k+1}$ a stationary point of subproblem:
\begin{align}
\label{eq:subpr} 
  \min_{y} m_{p}^q(y;x_k)
\end{align}
satisfying the following descent:
\begin{align}
\label{des:1}
  m_{p}^q(x_{k+1};x_k)\leq m_{p}^q(x_k;x_k) =  F(x_k).
\end{align}
\end{enumerate}
}  }
\end{center}

\medskip 

\noindent Our HO-DC algorithm is flexible in the sense that we can approximate both smooth terms $f$ and $g$ in the objective function of \eqref{eq:problem}  with higher-order Taylor approximations of different degrees $p$ and $q$, respectively.  Note that for $f\equiv 0$ and $q=1$ we recover  (proximal) DCA variants  proposed e.g., in  \cite{LePha:18,LeHuy:18,TaoLe:97,AnNam:17,BanBot:18}. Moreover, for $p=q=2$ HO-DC algorithm becomes a cubic regularized Newton method, see \cite{NesPol:06},  adapted to DC problems. It is also important to note that  in  our convergence analysis below we can relax the \textit{exact} stationary point condition for $x_{k+1}$ in subproblem \eqref{eq:subpr}, i.e., $0 \in \partial m_p^q(x_{k+1};x_k)$,   to  an \textit{approximate} stationary point condition of the form: 
$$\|  m^{x_{k+1}}  \| \leq \theta \| x_{k+1} - x_k \|^{\min(p,q)},$$ 
for some fixed parameter $\theta >0$, where $m^{x_{k+1}}   \in \partial m_p^q(x_{k+1};x_k)$.  For simplicity of the exposition however, we assume below that $x_{k+1}$ is a stationary point of the subproblem \eqref{eq:subpr}. 

\medskip

\noindent Let us also discuss about the \textit{implementability} of HO-DC algorithm. In order to get  a decreasing sequence $(F(x_k))_{k\geq 0}$, it is enough to assume that $x_{k+1}$ satisfies the descent \eqref{des:1}. However, to derive convergence rates for the sequence $(x_k)_{k\geq 0}$ to stationary points of problem \eqref{eq:problem}, we need to require additional properties for $x_{k+1}$, i.e., $x_{k+1}$ to be a stationary point of the subproblem \eqref{eq:subpr}. First, let us recall a well-known result stating that if the function  $f$ is  convex and $p>2$ differentiable,  having  the $p$ derivative Lipschitz  with constant $L_p^f$, then the regularized Taylor approximation \cite{Nes:20}:
	\begin{equation*} 
	y \mapsto  T_p^{f}(y;x) + \frac{M_p}{(p+1)!}\| y-x\|^{p+1} 
	\end{equation*}
is also a convex function  in   $y$ provided that $M_p \geq pL_p^{f}$.  In conclusion, based on Assumption \ref{as:1}.1 and choosing the regularization parameter $M_p$ appropriately, the first term in  $m_{p}^q(y;x)$, i.e., $T_p^f(y;x) + \frac{M_p}{(p+1)!}\|y - x\|^{p+1} + \psi(y)$,  is always convex in the first argument $y$ for any $p\geq 1$ (recall that for $p=1$ or $p=2$ the Taylor approximation of a convex function is always convex). As a consequence, for any $p \geq 1$ and $q=1$ our proposed approximation model   $m_{p}^q(\cdot;x)$ is convex, thus easy to minimize.  Of course,  the subproblem \eqref{eq:subpr}, which we need to solve in order to compute $x_{k+1}$, is usually  nonconvex for all $q \geq 2$.   We can show however  that one can still use the powerful tools from convex optimization to solve the \textit{nonconvex} subproblem \eqref{eq:subpr}, even globally, for other  choices of $p$ and $q$. More precisely, when  the Taylor approximations for $f$ and $g$ are of order $1$ or $2$,  and $\psi =0$, one can prove that the corresponding subproblem \eqref{eq:subpr} is equivalent to a convex one-dimensional optimization problem. Indeed, for $p=q=2$ or $p=2, q=1$ or $p=1, q=2$,  \textit{HO-DC algorithm becomes a cubic regularized Newton scheme for minimizing difference of convex functions} and the corresponding cubic regularized Newton step \eqref{eq:subpr} takes the simplified form (we denote $h=y-x_k$):
\begin{align}  
\label{eq:subproblem2}
\min_{h} \, \langle v, h\rangle + \frac{1}{2} \langle H h, h\rangle + \frac{M}{6} \|h\|^3,   
\end{align}
where $v \in \rset^n, H \in \rset^{n \times n}$ is a symmetric matrix and $M>0$. Then, the \textit{global minimum} of (possibly nonconvex) subproblem  \eqref{eq:subproblem2} can be computed as (see \cite{NesPol:06}(Section 5)):
\[  h^* = - \left( H + \frac{M r^*}{2} I_n\right)^{-1} v, \]
where  the convex set $\mathcal{D} = \left \{r \in \rset: \; H + \frac{M r}{2} I_n \succ 0, \; r \geq 0  \right \}$ and $r^*$ is the solution of  the convex one-dimensional optimization subproblem: 
\[ \min_{r \in \mathcal{D}}  \frac{1}{2}  \left \langle \left( H + \frac{M r}{2} I_n\right)^{-1} v, v \right \rangle  + \frac{M}{12} r^3.  \] 
There are many efficient numerical tools from convex optimization to solve this convex one-dimensional subproblem in $r$, e.g., interior point methods \cite{NesNem:94}. Alternatively, in the non-degenerate situation the solution of this 
subproblem  can be found from one-dimensional equation:
\[ r = \| ( H + \frac{M r}{2} I_n)^{-1} v \|, \;\; r \geq \frac{2}{M} \max(0, - \lambda_{\min}(H)).   \]
Several  technique for solving such equation were developed for 
trust region methods (see \cite{ConGou:00}(Chapter 7)).


\section{Convergence analysis for HO-DC algorithm}
\noindent In the next sections we derive a complete convergence analysis for our higher-order difference of convex functions  (HO-DC) algorithm. First, let us prove that the sequence $(F(x_k))_{k\geq 0}$ is nonincreasing.

\begin{theorem}
\label{th:1}
Let Assumption \ref{as:1} hold and let $(x_k)_{k\geq 0}$ be generated by HO-DC algorithm with $M_p > L_p^f$ and $M_q > L_q^g$. Then, we have:
\begin{enumerate}
\item  The sequence $(F(x_k))_{k\geq 0}$ satisfies the descent:
\begin{align}\label{eq:dec}
F(x_{k+1})  \leq F(x_k) - 2  \left( \frac{M_p - L_p^f}{(p+1)!} \cdot \frac{M_q - L_q^g}{(q+1)!} \right)^{\frac{1}{2}}   \|x_{k+1} - x_k \|^{\frac{p+q+2}{2}}. 
\end{align}
\item The sequence $(x_k)_{k\geq 0}$ satisfies:
\begin{align*}
    \sum_{i=0}^{\infty} \|x_{k+1} - x_k\|^{\frac{p+q+2}{2}} < \infty \;\;\text{ and } \; \lim_{k\to \infty} \|x_{k+1} - x_k\| = 0. 
\end{align*}
\end{enumerate}
\end{theorem} 

\begin{proof}
From the inequalities \eqref{eq:tpf}  and \eqref{eq:tqg} applied to  functions $f$ and $g$, we get:
\begin{align*}
  &  - \frac{L_p^f}{(p+1)!} \|x_{k+1} - x_k \|^{p+1} -  \frac{ L_q^g}{(q+1)!} \|x_{k+1} - x_k \|^{q+1} + F(x_{k+1}) \\ 
  & \leq T_p^f(x_{k+1};x_k) + \psi(x_{k+1}) -  T_q^g(x_{k+1};x_k).
\end{align*}
Further, adding on both sides of the previous inequality $\frac{M_p}{(p+1)!} \|x_{k+1} - x_k \|^{p+1} + \frac{M_q}{(q+1)!} \|x_{k+1} - x_k \|^{q+1}$ and using the descent \eqref{des:1}, we also get:
\begin{align*}
& \left( \frac{M_p - L_p^f}{(p+1)!} \|x_{k+1} - x_k \|^{p+1} + \frac{M_q - L_q^g}{(q+1)!} \|x_{k+1} - x_k \|^{q+1} \right)  + F(x_{k+1})\\
&\leq  m_p^q(x_{k+1};x_k)  \overset{\eqref{des:1}}{\leq} m_p^q(x_k;x_k) = F(x_k).
\end{align*}
\noindent Using  that   $2 (a b)^{1/2}\leq a + b$ for any  $a, b \geq 0$, we further get:
\begin{align}   
\label{ineq:ab}
& 2  \left( \frac{M_p - L_p^f}{(p+1)!} \cdot \frac{M_q - L_q^g}{(q+1)!} \right)^{\frac{1}{2}}   \|x_{k+1} - x_k \|^{\frac{p+q+2}{2}} \nonumber \\ 
& \leq  \frac{M_p - L_p^f}{(p+1)!} \|x_{k+1} - x_k \|^{p+1} + \frac{M_q - L_q^g}{(q+1)!} \|x_{k+1} - x_k \|^{q+1} \\
& \leq F(x_k) - F(x_{k+1}), \nonumber 
\end{align}
which yields the first statement \eqref{eq:dec}. Further, summing the last inequality from $i=0$ to $k-1$ and using that $F$ is bounded from below by  $F^*$ (see Assumption \ref{as:1}.4), we get:
\begin{align*}
   &  \sum_{i=0}^{k-1} 2  \left( \frac{M_p - L_p^f}{(p+1)!} \cdot \frac{M_q - L_q^g}{(q+1)!} \right)^{\frac{1}{2}}   \|x_{i+1} - x_i \|^{\frac{p+q+2}{2}} \\
    &\leq F(x_0) - F(x_k) \leq F(x_0) - F^* < \infty, 
\end{align*}
and taking $k \to \infty$ the second statement follows. \qed
\end{proof}

\medskip 

\noindent Let us denote:
\[  T_{p,q} (y;x) = T_p^f(y;x) + \frac{M_p}{(p+1)!}\|y - x\|^{p+1} -  T_q^g(y;x) + \frac{M_q}{(q+1)!}\|y - x\|^{q+1},  \]
and consequently,
\[  m_p^q(y;x) =  T_{p,q}(y;x) + \psi(y). \]
Moreover, from Theorem \ref{th:1} we have that $\| x_{k+1} - x_k\|$ converges to $0$, thus it is bounded, i.e., $\| x_{k+1} - x_k\| \leq C_x$ for all $k \geq 0$. Hence, let us  define:
\[ C_p^q = \max \left(  \frac{L_p^f + M_p}{p!} C_x^{p-q} +  \frac{L_q^g + M_q}{q!},  \frac{L_p^f + M_p}{p!}  +  \frac{L_q^g + M_q}{q!} C_x^{q-p}  \right).   \]

\begin{theorem}
\label{th:2}
Let Assumption \ref{as:1} hold and let $(x_k)_{k\geq 0}$ be generated by HO-DC algorithm with $M_p > L_p^f$ and $M_q > L_q^g$. Then, we have the following convergence rate in the minimum norm of subgradients of the objective function:

\begin{align}
\label{eq:rate1}
&  \min_{i=0:k-1}   S_F(x_{i})  = \min_{i=0:k-1}  \min_{F^{i} \in \partial F(x_{i})} \|F^{i}\| \\ 
&\leq \frac{C_p^q}{\left(2^2  \frac{M_p - L_p^f}{(p+1)!} \cdot \frac{M_q - L_q^g}{(q+1)!} \right)^{\frac{\min(p,q)}{p+q+2}}} \cdot \frac{(F(x_0) - F^*)^{\frac{2\min(p,q)}{p+q+2}}}{k^{\frac{2\min(p,q)}{p+q+2}}}. \nonumber 
\end{align}

\noindent Moreover, any limit point of the sequence $(x_k)_{k\geq 0}$ is a stationary point of problem \eqref{eq:problem}. If in addition, $F$ is coercive or the sequence  $(x_k)_{k\geq 0}$ is bounded, then there exits a subsequence of $(x_k)_{k\geq 0}$ which converges to a stationary point of problem \eqref{eq:problem}.
\end{theorem} 

\begin{proof}
Since $x_{k+1}$ is a stationary point of the subproblem  \eqref{eq:subpr},  then there exists $\zeta_{k+1} \in \partial \psi(x_{k+1})$ satisfying:
\begin{align}
\label{eq:zkp1}
 \nabla T_{p,q}(x_{k+1};x_k) + \zeta_{k+1} =0.   
\end{align}
Obviously, from basic calculus rules \cite{Mor:06}, we have that $$\nabla f(x_{k+1}) + \zeta_{k+1} - \nabla g(x_{k+1}) \in \partial F(x_{k+1}).$$  
On the other hand, from Assumption  \ref{as:1}  and inequality \eqref{eq:TayAppG1} we have:
\begin{align}
\label{eq:t1f}
\| \nabla f(x_{k+1}) - \nabla T_p^f(x_{k+1};x_k)\| \leq \frac{L_p^f}{p!} \|x_{k+1} - x_k\|^p
\end{align}
and 
\begin{align}
\label{eq:t1g}
\| \nabla g(x_{k+1}) - \nabla T_q^g(x_{k+1};x_k)\| \leq \frac{L_q^g}{q!} \|x_{k+1} - x_k\|^q. 
\end{align}
Then, combining the previous relations, we obtain:
\begin{align*}
 \min_{F^{k+1} \in \partial F(x_{k+1})} \|F^{k+1}\| & =  \min_{\zeta \in \partial \psi(x_{k+1})} \|\nabla f(x_{k+1}) + \zeta - \nabla g(x_{k+1}) \| \\
& \leq \| \nabla f(x_{k+1}) + \zeta_{k+1} - \nabla g(x_{k+1}) \|  \\
& \overset{\eqref{eq:zkp1}}{=}  \| \nabla f(x_{k+1})  -  \nabla T_{p,q}(x_{k+1};x_k) - \nabla g(x_{k+1}) \| \\
& \overset{\eqref{eq:t1f},\eqref{eq:t1g}}{\leq}   \frac{L_p^f + M_p}{p!}  \|x_{k+1} - x_k\|^p + \frac{L_q^g + M_q}{q!}  \|x_{k+1} - x_k\|^q. 
\end{align*}
Further,   if $p \geq q$ we have that:
\begin{align*}   
& \frac{L_p^f + M_p}{p!}  \|x_{k+1} - x_k\|^p + \frac{L_q^g + M_q}{q!}  \|x_{k+1} - x_k\|^q  \\
& \leq  \left(  \frac{L_p^f + M_p}{p!}  \|x_{k+1} - x_k\|^{p-q} +  \frac{L_q^g + M_q}{q!}  \right)  \|x_{k+1} - x_k\|^{q} \\
& \leq \left(  \frac{L_p^f + M_p}{p!} C_x^{p-q} +  \frac{L_q^g + M_q}{q!}  \right)  \|x_{k+1} - x_k\|^{q}. 
\end{align*}
Similarly, if $p < q$ we have that:
\begin{align*}   
& \frac{L_p^f + M_p}{p!}  \|x_{k+1} - x_k\|^p + \frac{L_q^g + M_q}{q!}  \|x_{k+1} - x_k\|^q  \\
& \leq \left(  \frac{L_p^f + M_p}{p!}  +  \frac{L_q^g + M_q}{q!} C_x^{q-p}  \right)  \|x_{k+1} - x_k\|^{p}. 
\end{align*}
Hence, using the definition of $C_p^q$, we further get:
\begin{align}
\label{eq:boundgrad}
  S_F(x_{k+1}) = \min_{F^{k+1} \in \partial F(x_{k+1})} \|F^{k+1}\| \leq  C_p^q \cdot \|x_{k+1} - x_k\|^{\min(p,q)}. 
\end{align}
Combining this relation with the descent  from Theorem \ref{th:1}, we further get:
\begin{align*}
&  S_F(x_{k+1}) = \min_{F^{k+1} \in \partial F(x_{k+1})} \|F^{k+1}\|^{\frac{p+q+2}{2\min(p,q)}} \\
& \overset{\eqref{eq:dec}}{\leq} 2^{-1} \left( \frac{M_p - L_p^f}{(p+1)!} \cdot \frac{M_q - L_q^g}{(q+1)!} \right)^{-\frac{1}{2}}  \left(C_p^q \right)^{\frac{p+q+2}{2\min(p,q)}} \cdot (F(x_k) - F(x_{k+1})). 
\end{align*}
Summing up this relation from $i=0$ to $i=k-1$ and using that $F$ is bounded from below by  $F^*$ (see Assumption \ref{as:1}.4), we get our first statement.  Further, from \eqref{eq:zkp1} we have that $-\nabla T_{p,q}(x_{k+1};x_k)  \in \partial \psi(x_{k+1})$. If $\bar x$ is a limit point of $x_k$, there exists a subsequence $(x_{k_j})_{j\geq 0}$ converging to $\bar x$. Then, since $x_{k_j} - x_{k_j-1} \to 0$ (see the second statement of Theorem \ref{th:1}), then also $x_{k_j-1} \to \bar x$,   and  since $\nabla T_{p,q}(\cdot;\cdot)$  is continuous in both arguments, we get that:
\[  -  \nabla T_{p,q}(x_{k_j};x_{k_j-1}) \to  -  \nabla T_{p,q}(\bar x;\bar x) = - (\nabla f(\bar x) -  \nabla g(\bar x)) \in  \partial \psi(\bar x), \]
thanks to the closedness of the graph of $\partial \psi$ and relation \eqref{eq:zkp1}. Thus, $\bar x$ is a stationary point of problem \eqref{eq:problem}. Moreover, by Theorem \ref{th:1}, the sequence $(F(x_k))_{k\geq 0}$ is decreasing and bounded from below (see Assumption \ref{as:1}.4), hence convergent to some finite value $F_* \geq F^*$. Therefore, when $F$ is coercive,  the sequence $(x_k)_{k\geq 0}$ must be bounded, which implies the rest of the claim.
\qed
\end{proof}

\noindent Note that if $p=q$ in Theorem \ref{th:2}   (see eq. \eqref{eq:rate1}) we recover  the usual global convergence rate $\mathcal{O}(k^{-\frac{p}{p+1}})$ for higher-order methods for solving ($p$ smooth) nonconvex optimization problems, see e.g., \cite{BiGaMa:17,cartis2017,NabNec:22,NecDan:20}, thus proving that our  convergence analysis is tight. \textit{To the best of our knowledge, the global convergence rate from Theorem \ref{th:2} is new even for $p=q=1$, i.e., for (proximal) DCA variants already studied in the literature \cite{LePha:18,LeHuy:18,TaoLe:97,AnNam:17,BanBot:18}.  Moreover, this theorem provides the first  worst-case complexity bound for the cubic regularized Newton type method (i.e., when $p=q=2$) in the context of DC programming. }


\section{Convergence analysis for HO-DC algorithm under KL}
Theorem \ref{th:2} shows  that any limit point of the sequence $(x_k)_{k\geq 0}$ generated by HO-DC algorithm is a stationary point of problem \eqref{eq:problem}. The objective in this section is to prove that under KL property of the objective function the whole sequence  $(x_k)_{k\geq 0}$ generated by HO-DC algorithm converges to a critical point of $F$. First, we summarize several properties of the limit point set. The set of all limit points of the sequence $(x_k)_{k\geq 0}$ generated by HO-DC algorithm from a starting point $x_0$ is denoted by $\Omega(x_0)$. 

\begin{lemma}
\label{lemma:kl}
Suppose that Assumption \ref{as:1} holds and  let $(x_k)_{k\geq 0}$ be a sequence generated by  HO-DC algorithm which is assumed to be bounded. Then, $\Omega(x_0)$ is a nonempty, compact and connected set contained in the set of stationary points of the objective function $F$ satisfying  $\lim_{k \to \infty} \text{dist}(x_k,\Omega(x_0))=0$. Moreover, assuming that either  objective $F$ is continuous or $x_{k+1}$ is a global minimum of  subproblem \eqref{eq:subpr}, then $F$ takes a  constant finite value $F_*$ on $\Omega(x_0)$.  
\end{lemma}

\begin{proof}
We have already proved in Theorem \ref{th:2} that any limit point of the sequence $(x_k)_{k\geq 0}$   is a stationary point of the objective function $F$. The set $\Omega(x_0)$ is  nonempty since $x_k$ is assumed to be bounded. Compactness and connectedness of $\Omega(x_0)$ follows from Lemma 3.5 in \cite{BolSab:14}. Further, the relation $\lim_{k \to \infty} \text{dist}(x_k,\Omega(x_0))=0$ follows from the definition of limit points.  Finally, from Theorem \ref{th:1} it follows that  the sequence $(F(x_k))_{k\geq 0}$ is decreasing and bounded from below by $F^*$ (see Assumption \ref{as:1}.4), hence it converges to some finite value $F_* \geq F^*$, i.e., $\lim_{k \to \infty} F(x_k)=F_*$. Let $\bar x \in  \Omega(x_0)$, then there is a subsequence $(x_{k_j})_{j\geq 0}$ of the sequence $(x_k)_{k\geq 0}$ such that $\lim_{j \to \infty} x_{k_j}=\bar x$. If $F$ is continuous, then obviously $F(\bar x) =F_*$, i.e., $F(\Omega(x_0))=F_*$. If $F$ is not continuous, then from Assumption \ref{as:1} it follows that $F$ is lower semicontinuous. Hence, $F_* = {\liminf}_{j \to \infty} F(x_{k_j}) \geq F(\bar x)$. On the other hand, if $x_{k+1}$ is a global minimum of the subproblem \eqref{eq:subpr}, then we have:  
\[  m_{p}^q(x_{k_j};x_{k_j-1}) \leq m_{p}^q(\bar x;x_{k_j-1}) \quad \forall j \geq 0 \]
and from Theorem \ref{th:1} we also have $\lim_{j \to \infty }x_{k_j} - x_{k_j-1} =0$ and consequently  $\lim_{j \to \infty} x_{k_j-1}=\bar x$. From these very reasons and from the  continuity of all (higher-order) derivatives of $f$ and $g$, taking $j \to \infty$ in the previous inequality we get:
$$  F_* =  \limsup_{j \to \infty}  F(x_{k_j}) =  \limsup_{j \to \infty} m_{p}^q(x_{k_j};x_{k_j-1}) \leq \limsup_{j \to \infty}  m_{p}^q(\bar x;x_{k_j-1}) = F(\bar x). $$
Combining the previous relations, we get $F_* = F(\bar x)$ for any $\bar x \in \Omega (x_0)$. 
\qed    
\end{proof}

\noindent Based on the previous lemma, now we are ready to prove the main result of this section. 
\begin{theorem}
\label{th:3}
Suppose that Assumption \ref{as:1} holds and  let $(x_k)_{k\geq 0}$ be a sequence generated by  HO-DC algorithm, assumed to be bounded and having the set of limit points $\Omega(x_0)$. Assume also that $F$ satisfies the  KL inequality on $\Omega(x_0)$. Then, the whole sequence $(x_k)_{k\geq 0}$ converges to a stationary point of the objective function $F$.  
\end{theorem}

\begin{proof}
In the light of \cite{AtBol:13,BolSab:14}, there is a general methodology  to prove that the whole sequence generated by a first order  algorithm converges to a critical point  under the  KL condition. We extend this methodology developed for first-order methods to our algorithm HO-DC. In our case we need to distinguish between two cases: $p \leq q$ and $p >q$. We prove only the first case  $p \leq q$, since the later case can be proved similarly. First, from the inequality \eqref{ineq:ab} the following  sufficient decrease property holds:
\begin{align}
\label{eq:kl1}
 \frac{M_p - L_p^f}{(p+1)!} \|x_{k+1} - x_k \|^{p+1}  \leq   F(x_k) - F(x_{k+1}).  
\end{align}   
Moreover, from inequality \eqref{eq:boundgrad}  we have the following  subgradient lower bound for the iterates gap (recall that we consider  $p \leq q$, hence $\min(p,q)=p$):
\begin{align}
\label{eq:kl2}
 S_F(x_{k+1}) \leq  C_p^q \cdot \|x_{k+1} - x_k\|^p.     
\end{align} 
Combining the previous two relations with the KL property of $F$,  we can show that the generated sequence $(x_k)_{k\geq 0}$ is a Cauchy sequence. Indeed, from the KL condition  \eqref{eq:kl_0} and Lemma \ref{lemma:kl} we have that there exists some $k_0$ such that:
\begin{equation}
\label{eq:kl3}
\kappa' (F(x_k) - F_*) \, C_p^q \, \|x_{k} - x_{k-1} \|^p  \overset{\eqref{eq:kl2}}{\geq} \! \kappa' (F(x_k) - F_*)\, S_{F}(x_k) \! \overset{\eqref{eq:kl_0}}{\geq}  1  \;\;\;    \forall k>k_0.
\end{equation}
From concavity of $\kappa$ we further get:
\begin{align*}
& \left( \kappa (F(x_k) - F_*) - \kappa (F(x_{k+1}) - F_*) \right) \, C_p^q \, \|x_{k} - x_{k-1} \|^p \\
& \geq  \kappa' (F(x_k) - F_*) (F(x_{k}) - F(x_{k+1})) \, C_p^q \, \|x_{k} - x_{k-1} \|^p \\
& \overset{\eqref{eq:kl3},\eqref{eq:kl1}}{\geq} \frac{M_p - L_p^f}{(p+1)!} \|x_{k+1} - x_k \|^{p+1} \quad   \forall k > k_0.
\end{align*}
Denoting $\Delta_{k,k+1} = \kappa (F(x_k) - F_*) - \kappa (F(x_{k+1}) - F_*)$ and $C= (C_{p}^q(p+1)!)/(M_p - L_p^f)$, the previous relation can be equivalently written as:
\[  C  \, \Delta_{k,k+1} \|x_{k} - x_{k-1} \|^p \geq  \|x_{k+1} - x_k \|^{p+1}  \quad   \forall k > k_0.  \]
Using the well-known relation that for any two positive scalars $a$ and $b$ we have $a^{\alpha_1} b^{\alpha_2} \leq \alpha_1 a + \alpha_2 b$ for any $\alpha_1,\alpha_2 \geq 0$ satisfying $\alpha_1+\alpha_2=1$ in the previous relation, we get:
\begin{align}
\label{eq:kl4}
 \|x_{k+1} - x_k \| &\leq (C \, \Delta_{k,k+1})^{\frac{1}{p+1}} \cdot \|x_{k} - x_{k-1} \|^{\frac{p}{p+1}} \nonumber \\
& \leq \frac{1}{p+1} C \, \Delta_{k,k+1} +  \frac{p}{p+1}  \|x_{k} - x_{k-1} \| \quad   \forall k > k_0.  
\end{align}
Summing \eqref{eq:kl4} from $k=k_0+1$ to some $K>k_0+1$ and using that the concave function $\kappa \geq 0$, we get:
\begin{align}
\label{eq:kl5}
\sum_{k=k_0+1}^K \|x_{k+1} - x_k \| & \leq C \, \kappa (F(x_{k_0+1}) - F_*)  +  p  \|x_{k_0+1} - x_{k_0} \|. 
\end{align}
The other case $p>q$ can be proved similarly, deriving a relation of the form:
\begin{align}
\label{eq:kl55}
\sum_{k=k_0+1}^K \|x_{k+1} - x_k \| & \leq C \, \kappa (F(x_{k_0+1}) - F_*)  +  q  \|x_{k_0+1} - x_{k_0} \|. 
\end{align}
Taking the limit as $K \to \infty$ we get  that the sequence $(x_k)_{k\geq 0}$ has finite length:
\[   \sum_{k=1}^\infty \|x_{k+1} - x_k \| < \infty.  \]
It is then clear that this implies that the sequence $(x_k)_{k\geq 0}$ is a Cauchy sequence and therefore it is convergent to some point $x^*$. Finally, Theorem \ref{th:2} shows that any limit point of $(x_k)_{k\geq 0}$   is a stationary point of the objective function $F$, hence $0 \in \partial F(x^*)$. This proves the statement of the theorem. 
\qed
\end{proof}

\noindent An important case of application of Theorem \ref{th:3} is when the objective function $F$ is semi-algebraic, i.e., it satisfies  the KL condition \eqref{eq:kl}. In this case we can derive explicit convergence rates for the sequence $(x_k)_{k\geq 0}$. 

\begin{theorem}
\label{th:4}
Let Assumption \ref{as:1} hold and   $(x_k)_{k\geq 0}$ be a sequence generated by  HO-DC algorithm. Assume also that $F$ is semi-algebraic (i.e., it satisfies  the KL condition \eqref{eq:kl} for some $r>1$). Then, the whole sequence $(x_k)_{k\geq 0}$ converges to a stationary point $x^*$ of the objective function $F$ with the following rates:\\  
(i) If $(r-1)\min(p,q) <1$ the convergence is sublinear. \\
(ii) If $(r-1)\min(p,q) \geq 1$ the convergence is linear.
\end{theorem}

\begin{proof}
From the proof of Theorem \ref{th:3} we know that $\zeta_k=\sum_{j=k}^\infty \|x_{j+1} - x_j \|$ is finite and $\zeta_k \to 0$ as $k \to \infty$. Since by the triangle inequality  $\|x_k - x^* \| \leq \zeta_k$, the rate of convergence of $x_k$ to $x^*$ can be inferred from the convergence rate of $\zeta_k$ to 0.
From \eqref{eq:kl5} and \eqref{eq:kl55} the following relation can be easily derived:
\begin{align*}
\sum_{j=k}^\infty \|x_{j+1} - x_j \| & \leq C \, \kappa (F(x_{k}) - F_*)  +  \min(p,q)  \|x_{k} - x_{k-1} \|, 
\end{align*}
which, using the expression $\kappa (t) = \sigma_r^{\frac{1}{r}} \frac{r}{r-1} t^{\frac{r-1}{r}}$, can be written compactly as:
\begin{align}
\label{eq:kl6}
\zeta_k & \leq C \, \kappa (F(x_{k}) - F_*)   +  \min(p,q) (\zeta_{k-1} - \zeta_{k}) \nonumber  \\
&= C \, \sigma_r^{\frac{1}{r}} \frac{r}{r-1} (F(x_{k}) - F_*)^{\frac{r-1}{r}}   +  \min(p,q) (\zeta_{k-1} - \zeta_{k}) \nonumber \\
& \overset{\eqref{eq:kl}}{\leq} C \, \sigma_r \frac{r}{r-1} (S_F(x_k))^{r-1} +  \min(p,q) (\zeta_{k-1} - \zeta_{k}) \nonumber \\ 
& \overset{\eqref{eq:boundgrad}}{\leq} C \, \sigma_r \frac{r}{r-1} \left( C_p^q \|x_k - x_{k_1}\|^{\min(p,q)}\right)^{r-1} +  \min(p,q) (\zeta_{k-1} - \zeta_{k}) \nonumber \\ 
&= C \, \sigma_r \frac{r}{r-1} \left( C_p^q \right)^{r-1}  (\zeta_{k-1} - \zeta_{k})^{(r-1)\min(p,q)} +  \min(p,q) (\zeta_{k-1} - \zeta_{k}) \nonumber \\ 
& =  \bar{C}  (\zeta_{k-1} - \zeta_{k})^{(r-1)\min(p,q)} +  \tilde{C} (\zeta_{k-1} - \zeta_{k}),
\end{align}
where $ \bar{C} = C \, \sigma_r \frac{r}{r-1} \left( C_p^q \right)^{r-1}$ and $\tilde{C} = \min(p,q)$. This type of recurrence  has been analysed extensively in the literature, see e.g., \cite{AtBol:13,AtBol:09,BolSab:14,NabNec:22}. Since $\zeta_k \to 0$ as $k \to \infty$ we can distinguish two cases:\\

\noindent \textit{Case 1}: If $(r-1)\min(p,q) <1$, then for $k$ sufficiently large and appropriate constant $C_1>0$, we can derive from \eqref{eq:kl6} the following simpler recurrence:
\[   \zeta_k  \leq C_1 (\zeta_{k-1} - \zeta_{k})^{(r-1)\min(p,q)} \]
which yields sublinear convergence rate of order (see e.g., Theorem 5 in \cite{AtBol:09})
\[  \zeta_k \leq \frac{c_1}{k^\frac{(r-1)\min(p,q)}{1 - (r-1)\min(p,q)}},  \]
with $c_1$ being a positive constant. \\

\noindent \textit{Case 2}: If $(r-1)\min(p,q) \geq 1$, then for $k$ sufficiently large and appropriate constant $C_2>0$, we can derive from \eqref{eq:kl6} the following simpler recurrence:
\[   \zeta_k  \leq C_2 (\zeta_{k-1} - \zeta_{k})\]
which yields linear convergence rate:
\[  \zeta_k \leq \frac{C_2}{1 + C_2} \zeta_{k-1},   \]
hence proving the statements of the theorem. 
\qed
\end{proof}

\noindent  The  convergence rates from Theorem \ref{th:4}  recover the complexity bounds  for (proximal) DCA variants (i.e., $p=q=1$)  derived in the literature, see e.g.,  \cite{LePha:18,LeHuy:18,TaoLe:97,AnNam:17,BanBot:18}.


\section{Adaptive HO-DC algorithm}
In some practical applications it may be difficult to estimate the Lipschitz constants $L_p^f$ and $L_q^g$ and thus difficult to choose  the constants $M_p$ and $M_q$ in HO-DC algorithm.  Hence, in this section, we propose an adaptive variant of HO-DC algorithm which is based on a line search procedure to choose these parameters. Since the surrogate model $m_p^q(\cdot;x)$ depends on the given constants $M_p$ and $M_q$, below we consider the following notation for the approximation model $m_p^q(y;x): = m_{M_p}^{M_q}(y;x)$ in order to reflect better this dependence. Note that the previous convergence results  are derived under Assumption \ref{as:1} and the sequence $(x_k)_{k\geq 0}$ generated by HO-DC algorithm having the following properties:
\begin{align}
\label{eq:adp0}
& x_{k+1} \; \text{is a stationary point of subproblem} \; \min_y m_{M_p}^{M_q}(y;x_k)
\end{align}

\noindent and a descent relation of the form (see \eqref{eq:dec}) 

\begin{align}\label{eq:adp}
F(x_{k+1}) \leq F(x_k) - \gamma \|x_{k+1} - x_k\|^{\frac{p+q+2}{2}},
\end{align}
where $\gamma>0$ is a given constant. Hence, in the following we propose an \textit{adaptive} higher-order DC algorithm, called (AHO-DC):   
\begin{center}
\noindent\fbox{%
\parbox{11.5cm}{%
\textbf{Algorithm AHO-DC}
\begin{enumerate}
\item[] Choose $x_{0}\in\text{dom}\;\psi$, $i=0$, $\gamma>0$ and $M_p^0, M_q^0 > 0$ 
\item[] For $k\geq 0$ \text{do}:
\item[] Step 1: compute $x_{k+1}$  satisfying  \eqref{eq:adp0} with $M_p = 2^{i}M_p^k$, $M_q = 2^{i}M_q^k$
\hspace*{1cm} If \eqref{eq:adp} \text{holds}, then go to Step 3
\item[] Step 2: else set $i=i+1$ and go to Step 1
\item[] Step 3: set $k=k+1$, $M_p^{k+1} = 2^{i-1}M_p^k$,  $M_q^{k+1} = 2^{i-1}M_q^k$ and $i=0$.
\end{enumerate}
}%
}
\end{center}

\medskip 

\noindent Note that step 1 in AHO-DC algorithm can be seen as a line search procedure: that is at each step $k \geq 0$ we choose $M_p^k$ and $M_q^k$, then build $m_{M_p^k}^{M_q^k}(y;x_k)$ and compute $x_{k+1}$ satisfying \eqref{eq:adp0}. If \eqref{eq:adp} doesn't hold, then we increase  $M_p^k \leftarrow  2\cdot M_p^k$,  $M_p^k \leftarrow  2\cdot M_q^k$ and recompute $m_{M_p^k}^{M_q^k}(y;x_k)$ using the new $M_p^k$  and $M_q^k$. We repeat this process until  condition \eqref{eq:adp} is satisfied. Note that this line search procedure finishes in a finite number of steps. Indeed, if $M_p^k \geq \gamma + L_p^{f}$ and $M_q^k \geq \gamma + L_q^{g}$, then from Theorem \ref{th:1} it follows that \eqref{eq:adp} holds. However, in practice the descent condition \eqref{eq:adp} may hold for much smaller values of $M_p^k$  and $M_q^k$ than the theory predicts.   Hence, using the same convergence analysis as in the previous sections, we can derive similar convergence rates as in Theorems  \ref{th:2}, \ref{th:3} and \ref{th:4} for the sequence $(x_k)_{k\geq 0}$ generated by AHO-DC algorithm.

\section{Conclusions}
In this paper we have developed (adaptive) higher-order algorithms for minimizing difference of convex functions (DC) with the first term in composite form. We have also showed the implementability of our algorithmic scheme, in particular, we have proposed the first cubic regularized Newton type algorithm in the context of DC programming.  Global convergence results and convergent rates were established under general assumptions but also under Kurdyka-Lojasiewicz inequality.


\end{document}